\documentclass[12pt]{article}
\usepackage{amssymb}
\usepackage{graphicx}
\usepackage{amsmath}

\RequirePackage[latin1]{inputenc} \RequirePackage[T1]{fontenc}
\newtheorem{theorem}{Theorem}[section]

\newtheorem{lemma}{Lemma}[section]

\newtheorem{corollary}{Corollary}[section]
\newtheorem{definition}{Definition}[section]
\newenvironment{proof}[1][Proof]{\textbf{#1.} }{\ \rule{0.5em}{0.5em}}
\def \N{\mbox{I\hspace{-.15em}N}}
\def \R{\mbox{I\hspace{-.15em}R}}
\def \P{\mbox{I\hspace{-.15em}P}}
\def \E{\mbox{I\hspace{-.15em}E}}

\begin{document}

\author{Auguste Aman\thanks{E-mail address: augusteaman5@yahoo.fr}\\
{\sl UFR de Math\'{e}matiques et Informatique},\\
{\sl 22 BP 582 Abidjan 22,\ C\^{o}te d'Ivoire}\\
}\date{}
\title{$L^{p}$-solutions of backward doubly stochastic differential
equations} \maketitle

\begin{abstract}
In this paper, our goal is solving backward doubly
stochastic differential equation (BDSDE for
short) under weak assumptions on the data. The first part of the paper is devoted to the
development of some new technical aspects of stochastic calculus related to BDSDEs. Then we
derive a priori estimates and prove existence and uniqueness of solutions, extending the results of Pardoux and Peng \cite{PP1} to the case where the solution is taked in $L^{p},\, p>1$ and the monotonicity conditions are satisfied. This study is limited to deterministic terminal time.
\end{abstract}

\section{Introduction}
In this paper, we are concerned with backward doubly stochastic differential equations
(BDSDEs for short in the remaining); a BDSDE is an equation of the following type:
\begin{eqnarray}
Y_t &=& \xi +\int^{ T}_{t}f(r, Y_r, Z_r) dr+\int_{t}^{T}g(r,Y_r,Z_r)\overleftarrow{dB}_r-\int^{T}_{t}Z_r dW_r ,\;  0\leq t\leq T.\label{a0}
\end{eqnarray}
This kind of equations has two different directions of stochastic integrals,
i.e., the equations involve both a standard (forward) stochastic Itô integral $dW_t$ and
a backward stochastic Itô integral $\overleftarrow{dB}_t$. $\xi$ is a random variable measurable with
respect to the past of $W$ up to time $T$ and it called the terminal condition. $f$ and $g$ are coefficients (also called generator) and, the unknowns are the adapted processes $\{Y_t\}_{t\in[0,T]}$ and $\{Z_t\}_{t\in[0,T]}$ with respect to the object $\sigma(W_r; 0\leq r\leq t)\vee\sigma(B_r-B_t; t\leq r\leq T)$, which is not a filtration.
Such equations, in the nonlinear case, have been introduced by Pardoux and Peng \cite{PP1}. They proved an existence and uniqueness result under the following assumption: $f$ and $g$ are Lipschitz continuous in both variables $y$ and $z$ and the data, $\xi$ and the processes $\{f(t, 0, 0)\}_{t\in[0,T]}$ and $\{g(t, 0, 0)\}_{t\in[0,T]}$, are square integrable. They also showed that BDSDEs are useful in probabilistic representations for solutions to some quasi-linear stochastic
partial differential equations (SPDEs). Since this first existence and uniqueness result, many papers have been devoted to
existence and/or uniqueness results under weaker assumptions. Among these papers, we can distinguish two different classes: scalar BDSDEs and multidimensional BDSDEs.
In the first case, one can take advantage of the comparison theorem: we refer to Shi et al. \cite{Shial} for this result. In this spirit, let us mention the contributions of N'zi and Owo \cite{NO}, which dealt with discontinuous coefficients. For multidimensional BDSDEs, there is no comparison theorem and to overcome this difficulty a monotonicity assumption on the generator $f$ in the variable y is used. This appear in the works of Peng and Shi \cite{PengShi}. More recently, N'zi and Owo \cite{NO1} established existence and uniqueness result under non-Lipschitz assumptions
\begin{eqnarray*}
|f(t,y,z)-f(t,y',z')|\leq \rho(t,|y-y'|^{2})+C\|z-z'\|^{2}\\
\|g(t,y,z)-g(t,y',z')\|\leq \rho(t, |y-y'|^{2})+\alpha\|z-z'\|^{2}
\end{eqnarray*}
where $C>$ and $0<\alpha<1$ are two constants and, $\rho$ a positive function satisfied some appropriate condtions.

Let us mention also that when the generator $f$ is monotone and continuous in $y$ and, $g\equiv 0,$ a result of
Pardoux et al. \cite{Pal}, provides the existence and uniqueness of a solution when the data $\xi$ and
$\{f(t, 0, 0)\}_{t\in[0,T]}$ are in $L^{p}$ even for $p\in (1, 2)$, both for equations
on a fixed and on a random time interval. This paper is devoted to
the generalization of this result to the case of $g\neq 0$ and a monotone generator $f$ for BDSDEs
only on a fixed time interval.

The paper is organized as follows: the next section contains all the notations and
some basic identities, while Section 3 contains essential estimates. Section 4 is devoted
to the main result that is existence and uniqueness of BDSDEs $(\ref{a0})$ where the data are in $L^{p}$ with $p\in(1, 2)$ on a fixed time interval.

\section{ Preliminaries}
\setcounter{theorem}{0} \setcounter{equation}{0}
\subsection{Assumptions and basic notations}
Let $(\Omega, \mathcal{F}, \P)$ be a probability space, and $T >0$ be fixed throughout this paper.
Let $\{W_t, 0\leq t\leq T\}$ and $\{B_t, 0\leq t\leq T\}$ be two mutually independent standard
Brownian motion processes, with values respectively in $\R^{d}$ and in $\R^{\ell}$, defined
on $(\Omega, \mathcal{F}, \P)$. Let $\mathcal{N}$ denote the class of $\P$-null sets of $\mathcal{F}$. For each $t\in[0, T]$, we define
\begin{eqnarray*}
\mathcal{F}_{t}=\mathcal{F}_{t}^{B}\otimes\mathcal{F}^{W}_{t,T}
\end{eqnarray*}
where for any process $\{\eta_t\},\, \mathcal{F}^{\eta}_{s,t}=\sigma\{\eta_{r}-\eta_{s},s\leq r\leq t\}\vee \mathcal{N},\; \mathcal{F}^{\eta}_t=\mathcal{F}^{\eta}_{0,t}$.

Note that the collection $\{\mathcal{F}_t, t\in[0, T]\}$ is neither increasing nor decreasing,
and it does not constitute a filtration. For any real $p>0$,
let us define the following spaces:\newline
$\mathcal{S}^{p}(\R^{n})$, denotes set of $\R^{n}$-valued, adapted
c\`{a}dl\`{a}g processes $
 \{X_{t}\}_{t\in \lbrack 0,T]}$ such that
\begin{equation*}
\|X\|_{\mathcal{S}^{p}}=\E\left( \sup\limits_{0\leq t\leq
T}|X_{t}|^{p}\right) ^{1\wedge \frac{1}{p}}<+\infty;
\end{equation*}
and $\mathcal{M}^{p}(\R^{n})$) is the set of  predictable processes
$\{X_{t}\}_{t\in \lbrack 0,T]}$  such that
\begin{eqnarray*}
\|X\|_{\mathcal{M}^{p}}=\E \left[ \left(
\int_{0}^{T}|X_{t}|^{2}dt\right) ^{\frac{p}{2}}\right] ^{1\wedge \frac{1}{p}%
}<+\infty .
\end{eqnarray*}
If $p\geq 1$, then $\|X\|_{\mathcal{S}^{p}}$ (resp
$\|X\|_{\mathcal{M}^{p}}$) is
a norm on $\mathcal{S}^{p}(\R^n)$ (resp. $\mathcal{M}%
^{p}(\R^{n})$) and these spaces are Banach spaces. But if p$\in
\left( 0,1\right) ,$ $\left( X,X^{\prime }\right) \longmapsto
\left\| X-X^{\prime }\right\| _{\mathcal{S}^{p}}$ (resp
$\left\|
X-X^{\prime }\right\| _{\mathcal{M}^{p}}$) defines a distance on $\mathcal{S}%
^{p}(\R^n),$ (resp. $\mathcal{M}^{p}(%
\R^{n})$) and under this metric, $\mathcal{S}^{p}(%
\R^n)$ (resp. $\mathcal{M}^{p}(\R^{n}))$ is complete.

Let
\begin{eqnarray*}
f:\Omega\times [0,T]\times\R^k
\times\R^{d\times k}\rightarrow \R^k; \ g:\Omega\times [0,T]\times\R^k
\times\R^{d\times k}\rightarrow \R^{\ell}
\end{eqnarray*}
be jointly measurable such that for any $(y,z)\in \R\times\R^{d\times k}$
\begin{description}
\item ({\bf H1})\, $f(.,y,z)\in M^{p}(0,T,\R^k),\,\,\ g(.,y,z)\in M^{p}(0,T,\R^{k\times\ell})$
\end{description}
Moreover, we assume that there exist constants $\lambda>0$  and
$0<\alpha<1$ such that for any $(\omega,t)\in \ \Omega\times[0,T];
(y_{1},z_{1}), (y_{2},z_{2})\in \R^k\times\R^{k\times d}$,
\begin{description}
\item ({\bf H1})
$
\left\{
\begin{array}{l}
(i)\; |f(t,y_{1},z_{1})-f(t,y_{2},z_{2})|^{2}\leq\lambda\|z_{1}-z_{2}\|^{2},\\\\
(ii)\;\langle y-y',f(t,y,z)-f(t,y',z)\rangle\leq\mu|y-y'|^{2},\\\\
(iii)\;\|g(t,y_{1},z_{1})-g(t,y_{2},z_{2})\|^{2}\leq\lambda|y_{1}-y_{2}|^{2}+\alpha\|z_{1}-z_{2}\|^{2}.
\end{array}\right.
$
\end{description}
Given a $\R^k$-valued $\mathcal{F}_{T}$-measurable random vector $\xi$,
we consider the following backward doubly stochastic
differential equation:
\begin{eqnarray}
Y_{t}=\xi+\int_{t
}^{T}f(s,Y_{s},Z_{s})ds+\int_{t}^{T}g(s,Y_{s},Z_{s})\,\overleftarrow{dB}_{s}-\int_{t}^{
T}Z_{s}\,dW_{s},\,\ 0\leq t\leq T.\label{a1}
\end{eqnarray}
We recall that the integral with respect to $\{B_t\}$ is a "backward Itô integral"
and the integral with respect to $\{W_t\}$ is a standard forward Itô integral. These
two types of integrals are particular cases of the Itô-Skorohod integral, see
Nualart and Pardoux \cite{NP}.
Before of all let us give meaning of a solution of backward doubly SDE.
\begin{definition}
A solution of backward doubly SDE $(\ref{a1})$ is a pair $(Y_{t},Z_{t})_{0\leq t \leq
T}$  of progressively measurable processes taking values in $\R^{k}
\times \R^{k\times d}$ such that: $\P$ a.s., $t\mapsto Z_t$ belongs in $L^{2}(0,T),\, t\mapsto f(t,Y_t,Z_t)$ belongs in $L^{1}(0,T)$ and satisfies $(\ref{a1})$.
\end{definition}

\subsection{A basic identity}
In the spirit of the works of Pardoux et al. (see \cite{Pal}), which treated the BSDE case i.e $g\equiv 0$, we want to deal with backward doubly SDEs with data in $L^{p}$, $p\in(1,2)$. That why, we start by a generalization to the multidimensional case of the Tanaka formula. Let us now introduce the notation
$\hat{x} = |x|^{-1}x{\bf 1}_{\{x=0\}}$. The following lemma will be our basic tool in the treatment of $L^{p}$-solutions and generalized Lemma 2.2 of \cite{Pal}.
\begin{lemma}
Let $\{K_t\}_{t\in[0,T]},\,\{H_t\}_{t\in[0,T]}$ and $\{G_{t}\}_{t\in[0,T]}$ be three progressively measurable processes
with values respectively in $\R^k, \R^{k\times d}$ and  $\R^{k\times l}$ such that $\P$-a.s.,
\begin{eqnarray*}
\int_{0}^{T}(K_s+|H_s|^{2}+|G_s|^{2})ds<\infty.
\end{eqnarray*}
We consider the $\R^{k}$-valued semi martingale $\{X_t\}_{t\in[0,T]}$ defined by
\begin{eqnarray}
Y_{t}=X_0+\int_{0}^{t}K_s\ ds+\int_{0}^{t}G_s\ \overleftarrow{dB}_{s}
+\int_{0}^{t}H_s\ dW_{s},\;\; 0\leq t\leq T.
\end{eqnarray}
Then, for any $p\geq 1$, we have
\begin{eqnarray*}
|X_{t}|^{p}-{\bf 1}_{\{p=1\}}L_t&=&|X_0 |^{p}+p\int_{0}^{t}|X_{s}|^{p-1}\langle\hat{X}_{s},K_s\rangle ds\\
&&+p\int_{0}^{t}|X_{s}|^{p-1}\langle \hat{X}_{s},G_{s}\overleftarrow{dB}_{s}\rangle+p\int_{0}^{t}|X_{s}|^{p-1}\langle \hat{X}_{s},H_{s}dW_{s}\rangle\\
&&-\frac{p}{2}\int_{0}^{t}|X_{s}|^{p-2}{\bf 1}_{\{X_s\neq 0\}}\{(2-p)(|G_s|^{2}-\langle \hat{X}_{s},G_sG_s^{*}\hat{X}_{s}\rangle)+(p-1)|G_s|^{2}\}ds\\
&&+\frac{p}{2}\int_{0}^{t}|X_{s}|^{p-2}{\bf 1}_{\{X_s\neq 0\}}\{(2-p)(|H_s|^{2}-\langle \hat{X}_{s},H_sH_s^{*}\hat{X}_{s}\rangle)+(p-1)|H_s|^{2}\}ds,
\end{eqnarray*}
where $\{L_t\}_{t\in[0,T]}$ is a continuous, increasing process with $L_0=0$, which increases only
on the boundary of the random set $\{t\in[0,T], \, X_t = 0\}$.
\end{lemma}
\begin{proof}
Since the function $x\mapsto|x|^p$ is not smooth enough (for $p\in[1,2)$), by the approximation argument used in \cite{Pal}, we set for $\varepsilon >0,\,
\begin{array}{l}
u_{\varepsilon}(x)=(|x|^{2}+\varepsilon^{2})^{1/2}.
\end{array}$
It is a smooth function and in virtue of Lemma 1.3 (e.g \cite{PP1}), we have
\begin{eqnarray*}
u_{\varepsilon}^{p}(X_{t})&=&u^{p}_{\varepsilon}(X_0)+p\int_{0}^{t}u_{\varepsilon}^{p-1}(X_{s})\langle X_{s},K_s\rangle ds\\
&&+p\int_{0}^{t}u_{\varepsilon}^{p-1}(X_{s})\langle X_{s},G_{s}\overleftarrow{dB}_{s}\rangle+p\int_{0}^{t}u_{\varepsilon}^{p-1}(X_{s})\langle X_{s},H_{s}dW_{s}\rangle\\
&&-\frac{1}{2}\int_{0}^{t}trace(D^{2}u_{\varepsilon}^{p}(X_{s})G_sG_s^{*})ds+\frac{1}{2}\int_{0}^{t}trace(D^{2}u_{\varepsilon}^{p}(X_{s})H_sH_s^{*})ds.
\end{eqnarray*}
The rest of the proof follows identically as lemma 2.2 proof's so that we omit it.
\end{proof}
\begin{corollary}
If $(Y,Z)$ is a solution of the BSDE $(\ref{a1}),\, p \geq 1,\, c(p)=p[(p-1)\wedge 1]/2$ and $0\leq t\leq T$, then
\begin{eqnarray*}
&&|Y_{t}|^{p}+c(p)\int_{t}^{T}|Y_s|^{p-2}{\bf 1}_{\{Y_s\neq 0\}}|Z_s|^{2}ds\\
&\leq &|Y_T |^{p}+p\int_{t}^{T}|Y_{s}|^{p-1}\langle\hat{Y}_{s},f(s,Y_s,Z_s)\rangle ds\\
&&+c(p)\int_{t}^{T}|Y_s|^{p-2}{\bf 1}_{\{Y_s\neq 0\}}|g(s,Y_s,Z_s)|^{2}ds\\
&&+p\int_{t}^{T}|Y_{s}|^{p-1}\langle \hat{Y}_{s},g(s,Y_{s},Z_s)\overleftarrow{dB}_{s}\rangle-p\int_{0}^{t}|Y_{s}|^{p-1}\langle \hat{Y}_{s},Z_{s}dW_{s}\rangle.
\end{eqnarray*}
\end{corollary}
\begin{proof}
As consequence of Lemma 2.1, for $0\leq t\leq T$ and $c(p)=p[(p-1)\wedge 1]/2$, we get
\begin{eqnarray*}
|X_{T}|^{p}&\geq&|X_t |^{p}+p\int_{t}^{T}|X_{s}|^{p-1}\langle\hat{X}_{s},K_s\rangle ds\\
&&+p\int_{t}^{T}|X_{s}|^{p-1}\langle \hat{X}_{s},G_{s}\overleftarrow{dB}_{s}\rangle+p\int_{t}^{T}|X_{s}|^{p-1}\langle \hat{X}_{s},H_{s}dW_{s}\rangle\\
&&-c(p)\int_{t}^{T}|X_{s}|^{p-2}{\bf 1}_{\{X_s\neq 0\}}|G_s|^{2}ds\\
&&+c(p)\int_{t}^{T}|X_{s}|^{p-2}{\bf 1}_{\{X_s\neq 0\}}|H_s|^{2}ds.
\end{eqnarray*}
Replace $(X,H)$ by $(Y,Z)$ the solution of backward doubly SDE $(\ref{a1})$, setting $K=f(.,Y,Z)$ and $G=g(.,Y,Z)$, we get the result.
\end{proof}

\section{Apriori estimates}

We give now an estimate which permit to control the process
$Z$ with the data and the process $Y$.

\begin{lemma}
Let assumptions $\left({\bf H1}\right)$-$\left({\bf H2 }\right)$ hold and let $\left( Y,Z\right) $ be a solution of backward doubly SDE $(\ref{a1})$. If $Y\in \mathcal{S}^{p}$ then $Z$ belong to $\mathcal{M}%
^{p}$ and there exists a real constant $C_{p,\lambda}$ depending only on $p$ and $\lambda$
such that,
\begin{eqnarray*}
\E\left[ \left( \int_{0}^{T}|Z_{r}|^{2}dr\right)
^{p/2}\right]
  &\leq &C_{p}%
\E\left\{ \sup_{0\leq t\leq T}|Y_{t}|^{p}+\left(
\int_{0}^{T}|f^{0}_{r}|dr\right)^{p}+ \left( \int_{0}^{T}|g^{0}_{r}|^{2}dr\right)
^{p/2}\right\}.
\end{eqnarray*}
\end{lemma}
\begin{proof}
Let $a$ be a real constant and for each integer $n$ let us define:
\begin{eqnarray*}
\tau_{n}=\inf \left \{t\in [0,T],\int_{0}^{t}|Z_{r}|dr \geq n \right \}
 \wedge T.
\end{eqnarray*}
The sequence $(\tau_n)_{n\geq 0}$ is of stationary type since the process $Z$ belongs to $\mathcal{M}^{p}$ and then $\int^{T}_{0} |Z_s|^{2}ds <\infty,\, \P$- a.s..
Next we use It\^o's formula to give
\begin{eqnarray}
&&|Y_{0}|^{2}+\int_{0}^{\tau_{n}}e^{ar}|Z_{r}|^{2}dr\nonumber\\ &=&
e^{a\tau_n}|Y_{\tau_{n}}|^{2}+2\int_{0}^{\tau_{n}}e^{ar}\langle Y_{r}, f(r,Y_{r},Z_{r})-aY_r\rangle dr
 + \int_{0}^{\tau_{n}} e^{ar}|g(r,Y_{r},Z_r)|^{2}dr\nonumber \\
&& + 2\int_{0}^{\tau_{n}}e^{ar}\langle Y_{r},g(r,Y_r,Z_r) \overleftarrow{dB}_{r}\rangle - 2\int_{0}^{\tau_{n}}
e^{ar}\langle Y_{r}, Z_{r} dW_{r}\rangle.\label{estZ}
\end{eqnarray}
But, from  assumption on $f$ together the standard inequality $2ab\leq \frac{1}{\varepsilon}a^{2}+\varepsilon b^{2}$, for any $\varepsilon$, we have:
\begin{eqnarray*}
2\langle Y_{r}, f(r,Y_{r},Z_{r})-aY_r\rangle&\leq& 2|Y_r||f^{0}_r|+2\mu|Y_r|^{2}+2\lambda|Y_r||Z_r|-a|Y_r|^{2}\\\\
&\leq &2|Y_r||f^{0}_r|+(2\mu+2\lambda+\varepsilon^{-1}\lambda^{2}-a)|Y_r|^{2}+\varepsilon|Z_r|^{2}.
\end{eqnarray*}
Moreover, from the assumption on $g$, we use again the above standard inequality to get
\begin{eqnarray*}
\|g(r,Y_r,Y_r)\|^{2}\leq (1+\varepsilon')\lambda|Y_r|^{2}+(1+\varepsilon')\alpha|Z_r|^{2}+(1+\frac{1}{\varepsilon'})|g^{0}_r|^{2}.
\end{eqnarray*}
Plugging this two last inequalities in $(\ref{estZ})$, we obtain:
\begin{eqnarray*}
&&|Y_{0}|^{2}+\int_{0}^{\tau_{n}}e^{ar}|Z_{r}|^{2}dr\nonumber\\ &\leq&
e^{a\tau_n}|Y_{\tau_{n}}|^{2}+[2\mu+(3+\varepsilon')\lambda+\varepsilon^{-1}\lambda^{2}-a]\int_{0}^{\tau_{n}}e^{ar}|Y_r|^{2}dr\\
&& + [\varepsilon+(1+\varepsilon')\alpha]\int_{0}^{\tau_{n}} e^{ar}|Z_r|^{2}dr+2\int_{0}^{\tau_{n}} e^{ar}|Y_r||f^{0}_r|dr+(1+\frac{1}{\varepsilon'})\int_{0}^{\tau_{n}} e^{ar}|g^{0}_r|^{2}dr\nonumber \\
&& + 2\int_{0}^{\tau_{n}}e^{ar}\langle Y_{r},g(r,Y_r,Z_r) \overleftarrow{dB}_{r}\rangle - 2\int_{0}^{\tau_{n}}
e^{ar}\langle Y_{r}, Z_{r} dW_{r}\rangle.\label{estZ}
\end{eqnarray*}
Choosing now $\varepsilon$ and $\varepsilon'$ small enough such that $\varepsilon+(1+\varepsilon')\alpha<1$ and $a$ such that\newline $2\mu+(3+\varepsilon')\lambda+\varepsilon^{-1}\lambda^{2}-a\leq 0$, we derive
\begin{eqnarray*}
\left(\int_{0}^{\tau_{n}}|Z_{r}|^{2}dr\right)^{p/2} &\leq&
C_{p,\lambda}\left\{\sup_{0\leq t\leq\tau_{n}}|Y_{t}|^{p}+\left(\int_{0}^{\tau_{n}}|f^{0}_{r}|^{2} dr\right)^{p/2}+\left(\int_{0}^{\tau_{n}}|g^{0}_{r}|^{2}dr\right)^{p/2}\right.\\
&&\left.+\left|\int_{0}^{\tau_{n}}e^{\alpha r}\langle Y_{r}, g(s,Y_r,Z_{r})\overleftarrow{dB}_{r}\rangle\right|^{p/2}+\left|\int_{0}^{\tau_{n}}e^{\alpha r}\langle Y_{r}, Z_{r}dW_{r}\rangle\right|^{p/2}\right\}. \label{a2}
\end{eqnarray*}
Next thanks to BDG's inequality we have:
\begin{eqnarray*}
\E\left(\left|\int_{0}^{\tau_{n}}e^{\alpha r}\langle Y_{r},Z_{r}dW_{r}\rangle
\right|^{p/2}\right)&\leq& d_{p}\E\left[ \left(\int_{0}^{\tau_{n}}
|Y_{r}|^{2}|Z_{r}|^{2}dr\right)^{p/4}\right]\nonumber \\
&\leq& \bar{C}_{p}\E\left[ \sup_{0\leq t\leq \tau_{n}}|Y_{t}|^{p/2}\left(\int_{0}^{\tau_{n}}
|Z_{r}|^{2}dr\right)^{p/4}\right]\nonumber\\
&\leq& \frac{\bar{C}_{p}^{2}}{\eta_1}\E\left(\sup_{0\leq t\leq \tau_{n}}|Y_{t}|^{p}\right)+
\eta_1\E\left(\int_{0}^{\tau_{n}}|Z_{r}|^{2}dr\right)^{p/2}.
\label{a3}
\end{eqnarray*}
and
\begin{eqnarray*}
\E\left(\left|\int_{0}^{\tau_{n}}e^{\alpha r}\langle Y_{r}, g(s,Y_r,Z_{r})\overleftarrow{dB}_{r}\rangle\right|^{p/2}\right)&\leq& d_{p}\E\left[ \left(\int_{0}^{\tau_{n}}
|Y_{r}|^{2}|g(r,Y_r,Z_{r})|^{2}dr\right)^{p/4}\right]\nonumber \\
&\leq& \bar{C}_{p}\E\left[ \sup_{0\leq t\leq \tau_{n}}|Y_{t}|^{p/2}\left(\int_{0}^{\tau_{n}}
|g(r,Y_r,Z_{r})|^{2}dr\right)^{p/4}\right]\nonumber\\
&\leq& \frac{\bar{C}_{p}^{2}}{\eta_2}\E\left(\sup_{0\leq t\leq \tau_{n}}|Y_{t}|^{p}\right)+
\eta_2\E\left(\int_{0}^{\tau_{n}}|g(r,Y_r,Z_{r})|^{2}dr\right)^{p/2}\\
&\leq&C_p\E\left(\sup_{0\leq t\leq \tau_{n}}|Y_{t}|^{p}+\left(\int_{0}^{\tau_n}|g^{0}_r|^{2}\right)^{p/2}\right)\\
&&+(1+\eta_3)\eta_2\alpha\E\left(\int_{0}^{\tau_{n}}|Z_{r}|^{2}dr\right)^{p/2}.
\end{eqnarray*}
Finally plugging the two last inequalities in the previous one, choosing $\eta_1, \eta_2$ and $\eta_3$ small enough, we
finally use Fatou's Lemma to obtain the desired result.
\end{proof}

We keep on this study by stating the standard estimate in our context. The difficulty
comes from the fact that $f$ is not Lipschitz in $y$ and also from the fact that the function
$y\mapsto|y|^p$ is not $\mathcal{C}^{2}$ since we will work with $p\in(1,2)$.
\begin{lemma}
Assume $\left(
\text{\textbf{H1}}\right)$-$\left({ \bf H2}\right)$. Let
$\left(Y,Z\right)$ be a solution of the backward doubly SDE
associated to the data $(\xi,f,g)$ where $Y$ belong to
$\mathcal{S}^{p}$. Then there exists a constant $C_{p,\lambda}$ depending
only on $p$ and $\lambda$ such that
\begin{eqnarray*}
\E\left\{\sup_{0\leq t \leq T}|Y_{t}|^{p}+
\left(\int_{0}^{T}|Z_{s}|^{2}ds\right)^{p/2}
\right\}&\leq& C_{p,\lambda}\E\left\{|\xi|^{p}+\left(\int_{0}^{T}|f^{0}_{s}|ds\right)^{p}\right.\\
&&+\left.\left(\int_{0}^{T}|g^{0}_{s}|^{2}ds\right)^{p/2}+\int_{0}^{T}|Y_s|^{p-2}{\bf 1}_{\{Y_s\neq 0\}}|g^{0}_{s}|^{2}ds\right\}.
\end{eqnarray*}
\end{lemma}
\begin{proof}
From Corollary 2.1 for any $a>0$ and any $0\leq t\leq u\leq T$ we have
\begin{eqnarray*}
&&e^{apt}|Y_{t}|^{p}+
c(p)\int_{t}^{u}e^{aps}|Y_s|^{p-2}{\bf 1}_{\{Y_s\neq\ 0\}}|Z_{s}|^{2}ds\\
&\leq& e^{apu}|Y_u|^{p}-ap\int_{u}^{T}e^{aps}|Y_s|^{p}ds+p\int_{t}^{u}e^{aps}|Y_s|^{p-1}\langle\widehat{Y}_s,f(s,Y_s,Z_s)\rangle ds\\
&&+c(p)\int_{t}^{u}e^{aps}|Y_s|^{p-2}{\bf 1}_{\{Y_s\neq\ 0\}}|g(s,Y_s,Z_s)|^{2}ds+p\int_{t}^{u}e^{aps}|Y_s|^{p-1}\langle \widehat{Y}_s,g(s,Y_s,Z_s)\overleftarrow{dB}_s\rangle\\
&&-p\int_{t}^{u}e^{aps}|Y_s|^{p-1}\langle\widehat{Y}_s,Z_sdW_s\rangle.
\end{eqnarray*}
The assumption on $f$ and $g$ yields the inequalities
\begin{eqnarray}
\langle \hat{y},f(s,y,z)\rangle&\leq& |f_{s}^{0}| + \mu|y| + \lambda|z|\nonumber\\
&& \mbox{and}\label{A}\\
\|g(s,y,z)\|^{2}&\leq& (1+\varepsilon)\lambda|y|^{2}+(1+\varepsilon)\alpha|z|^{2}+(1+\frac{1}{\varepsilon})|g^{0}_s|^{2},\nonumber
\end{eqnarray}
for $\varepsilon>0$, from which we deduce that with probability one, for all $t\in[0,T]$,
\begin{eqnarray*}
&&e^{apt}|Y_{t}|^{p}+
c(p)\int_{t}^{u}e^{aps}|Y_s|^{p-2}{\bf 1}_{\{Y_s\neq\ 0\}}|Z_{s}|^{2}ds\\
&\leq& e^{apu}|Y_u|^{p}+[p(\mu-a)+c(p)(1+\varepsilon)\lambda]\int_{u}^{T}e^{aps}|Y_s|^{p}ds\\
&&+p\int_{t}^{u}e^{aps}|Y_s|^{p-1}|f^{0}_s|ds+c(p)(1+\varepsilon^{-1})\int_{t}^{u}e^{aps}|Y_s|^{p-2}{\bf 1}_{\{Y_s\neq\ 0\}}|g_{s}^{0}|^{2}ds\\
&&+c(p)(1+\varepsilon)\alpha\int_{t}^{u}e^{aps}|Y_s|^{p-2}{\bf 1}_{\{Y_s\neq\ 0\}}|Z_s|^{2}ds+p\lambda\int_{t}^{u}e^{aps}|Y_s|^{p-1}|Z_s|ds
\\
&&+p\int_{t}^{u}e^{aps}|Y_s|^{p-1}\langle \widehat{Y}_s,g(s,Y_s,Z_s)\overleftarrow{dB}_s\rangle-p\int_{t}^{u}e^{p\alpha s}|Y_s|^{p-1}\langle \widehat{Y}_s,Z_sdW_s\rangle .
\end{eqnarray*}
First of all we deduce from the previous inequality that, $\P$-a.s.,
\begin{eqnarray*}
\int_{0}^{T}e^{aps}|Y_s|^{p-2}{\bf 1}_{\{Y_s\neq\ 0\}}|Z_{s}|^{2}ds<\infty.
\end{eqnarray*}
Moreover, use again the standard algebraic inequality we have
\begin{eqnarray*}
p\lambda|Y_s|^{p-1}|Z_s|\leq \gamma^{-1}\frac{p\lambda^{2}}{2[(p-1)\wedge 1]}|Y_s|^{p}+\gamma c(p)|Y_s|^{p-1}{\bf 1}_{\{Y_s\neq 0\}}|Z_s|^{2},
\end{eqnarray*}
for any $\gamma$. Thus we take $\varepsilon, \gamma$ small enough such that $\alpha'=1-[(1+\varepsilon)\alpha+\gamma]>0$ and set $T=u$ to derive:
\begin{eqnarray}
&&+\alpha'c(p)\int_{t}^{T}e^{aps}|Y_s|^{p-2}{\bf 1}_{\{Y_s\neq\ 0\}}|Z_{s}|^{2}ds\nonumber\\
&\leq& e^{apT}|\xi|^{p}+p\int_{t}^{T}e^{aps}|Y_s|^{p-1}|f^{0}_s|ds\nonumber\\
&&+c(a,\varepsilon,\gamma)\int_{t}^{T}e^{aps}|Y_s|^{p}ds+c(p)(1+\varepsilon^{-1})\int_{t}^{T}e^{aps}|Y_s|^{p-2}{\bf 1}_{\{Y_s\neq\ 0\}}|g_{s}^{0}|^{2}ds\nonumber\\
&&+p\int_{t}^{T}e^{aps}|Y_s|^{p-1}\langle \widehat{Y}_s,g(s,Y_s,Z_s)\overleftarrow{dB}_s\rangle
-p\int_{t}^{T}e^{p\alpha s}|Y_s|^{p-1}\langle \widehat{Y}_s,Z_sdW_s\rangle,\label{estYZ3}
\end{eqnarray}
where $c(a,\varepsilon,\gamma)=\mu+[(1+\varepsilon)\lambda+\gamma^{-1}\lambda^{2}]/2[(p-1)\wedge 1]-a$.

Let us set
\begin{eqnarray*}
X &=&e^{apT}|\xi|^{p}+p\int_{0}^{T}e^{aps}|Y_s|^{p-1}|f^{0}_s|ds+c(p)(1+\varepsilon^{-1})\int_{0}^{T}e^{aps}|Y_s|^{p-2}{\bf 1}_{\{Y_s\neq\ 0\}}|g_{s}^{0}|^{2}ds,
\end{eqnarray*}
then, taking the expectation, the use of Gronwall lemma give us
\begin{eqnarray*}
\E\left(e^{apt}|Y_{t}|^{p}\right)\leq C_p \E(X)
\end{eqnarray*}
because, from BDG inequality, on can show that $M_t=\{\int_{t}^{T}e^{aps}|Y_s|^{p-1}\langle \widehat{Y}_s,g(s,Y_s,Z_s)\overleftarrow{dB}_s\rangle\}$ and $N_t=\{\int_{t}^{T}e^{p\alpha s}|Y_s|^{p-1}\langle \widehat{Y}_s,Z_sdW_s\rangle\}$ are respectively uniformly integrable martingale.
Coming back to inequality $(\ref{estYZ3})$, taking the expectation, for t = 0, we get
\begin{eqnarray}
&&\alpha'c(p)\E\int_{0}^{T}e^{aps}|Y_s|^{p-2}{\bf 1}_{\{Y_s\neq\ 0\}}|Z_{s}|^{2}ds\leq C_p\E(X)\label{estZY1}
\end{eqnarray}
and
\begin{eqnarray}
&&\E\left(\sup_{0\leq t\leq T}e^{apt}|Y_{t}|^{p}\right)\leq \E(X)+k_p\E\langle M,M\rangle_T^{1/2}+h_p\E\langle N,N\rangle_T^{1/2}.\label{estYZ2}
\end{eqnarray}
But, we have
\begin{eqnarray*}
h_p\E\langle N,N\rangle_T^{1/2}\leq \frac{1}{4}\E\left(\sup_{0\leq t\leq T}e^{apt}|Y_{t}|^{p}\right)+4h^{2}_p\E\int_{0}^{T}e^{aps}|Y_s|^{p-2}{\bf 1}_{\{Y_s\neq\ 0\}}|Z_{s}|^{2}ds
\end{eqnarray*}
and
\begin{eqnarray*}
k_p\E\langle M,M\rangle_T^{1/2}&\leq& \frac{1}{4}\E\left(\sup_{0\leq t\leq T}e^{apt}|Y_{t}|^{p}\right)+4k^{2}_p\E\int_{0}^{T}e^{aps}|Y_s|^{p-2}{\bf 1}_{\{Y_s\neq\ 0\}}|g(s,Y_s,Z_s)|^{2}ds\\
&\leq&\frac{1}{4}\E\left(\sup_{0\leq t\leq T}e^{apt}|Y_{t}|^{p}\right)+4h^{2}_p\E\int_{0}^{T}e^{aps}|Y_s|^{p-2}{\bf 1}_{\{Y_s\neq\ 0\}}|Z_{s}|^{2}ds\\
&&+d_p\E\left(X\right).
\end{eqnarray*}
Coming back to inequalities $(\ref{estZY1})$ and $(\ref{estYZ2})$,  we obtain
\begin{eqnarray*}
\E\left(\sup_{0\leq t\leq T}e^{apt}|Y_{t}|^{p}\right)\leq C_p\E(X)
\end{eqnarray*}
Applying once again Young's inequality, we get
\begin{eqnarray*}
pC_p\int_{0}^{T}e^{p\alpha s}|Y_s|^{p-1}|f^{0}_s|ds&\leq &\frac{1}{2}\sup_{0\leq s\leq T}|Y_s|^{p}+C'_p\left(\int_{0}^{T}e^{p\alpha s}|f^{0}_s|ds\right)^{p}
\end{eqnarray*}
from which we deduce, coming back to the definition of $X$, that
\begin{eqnarray*}
\E\left(\sup_{0\leq t\leq T}e^{apt}|Y_{t}|^{p}\right)\leq C_p\E\left[|\xi|^{p}+\left(\int_{0}^{T}e^{p\alpha s}|f^{0}_s|ds\right)^{p}
+\int_{0}^{T}e^{aps}|Y_s|^{p-2}{\bf 1}_{\{Y_s\neq\ 0\}}|g^{0}_{s}|^{2}ds\right].
\end{eqnarray*}
The result follows from Lemma 3.1.
\end{proof}

\section{Existence and uniqueness of a solution}
\setcounter{theorem}{0} \setcounter{equation}{0}
In this section we prove existence and uniqueness result for the backward doubly SDE associated to data $(\xi,f,g)$ in
$L^{p}$, with the help of $L^{\infty}$-approximation and priori estimates given above.

In addition to the previous hypothesis, we will work under the following assumptions: for some $p>1$,
\begin{description}
\item ({\bf H3})
$
\left\{
\begin{array}{l}
(i)\;\E\left[|\xi|^{p}\right]<\infty,\\\\
(ii)\;\P\, a.s.\; \forall\, (t,z)\in[0,T]\times\R^{k\times d},\; y\mapsto f(t,y,z)\, \mbox{is continuous},\\\\ 
(iii)\; g(.,0,0)\equiv 0,\\\\
(iv)\;\forall\, r>0, \; \psi_r(t)=\sup_{|y|<r}|f(t,y,0)-f_{t}^{0}|\in L^{1}([0,T]\time\Omega,m\otimes\P).
\end{array}\right.
$
\end{description}
We want to obtain an existence and uniqueness result for backward doubly SDE $(\ref{a1})$ under the
previous assumptions for all $p>0$.

Firstly, let us give this result, that in our mind extended the result of Pardoux and Peng (see Theorem 1.1, \cite{PP1}). Indeed, here the coefficient $f$ is supposed to be non-Lipschitz in $y$ but monotonic. For this, let us introduce the following assumption:
\begin{description}
\item ({\bf H4})\;$\P\, a.s.\; \forall\, (t,y)\in[0,T]\times\R^{k},\, |f(t,y,0)|\leq |f(t,0,0)|+\varphi(|y|)$,
\end{description}
where $\varphi:\R_+\rightarrow\R_+$ is a deterministic continuous increasing function.
\begin{theorem}
Let $p=2$. Under assumptions $({\bf H1)}$-$({\bf H4)}$, BSDE $(\ref{a1})$ has a
unique solution in $S^2 \times M^2$.
\end{theorem}
\begin{proof}
It follows easily by combining argument of Pardoux (see Theorem 2.2 \cite{P}) with one used in Pardoux and Peng (see Theorem 1.1, \cite{PP1})
\end{proof}

We now prove our existence and uniqueness result.
\begin{theorem}
Under assumptions $({\bf H1)}$-$({\bf H3)}$, BSDE $(\ref{a1})$ has a
unique solution in $S^p \times M^p$.
\end{theorem}
\begin{proof}
{\bf Uniqueness}
\newline
Let us consider $(Y,Z)$ and $(Y',Z')$ two solutions of backward SDE with data $\left(\xi,f,g \right)$ in the appropriate space. We denote by $(U, V)$ the process
$(Y - Y' , Z - Z')$; we show easily that this process is solution to the following backward doubly SDE:
\begin{eqnarray*}
U_t=\int_{t}^{T}h(s,U_s,V_s)ds+\int_{t}^{T}k(s,U_s,V_s)\overleftarrow{dB}_s-\int_{t}^{T}V_sdW_s,\;\; 0\leq t\leq T,
\end{eqnarray*}
where $h$ and $k$ stand the random functions
\begin{eqnarray*}
h(s,y,z)=f(s,y+Y'_s,z+Z'_s)-f(s,Y'_s,Z'_s),
\\
k(s,y,z)=g(s,y+Y'_s,z+Z'_s)-g(s,Y'_s,Z'_s).
\end{eqnarray*}
Thanks to assumptions $({\bf H2})$, functions $h$ and $k$ satisfy assumption $(\ref{A})$ with
respectively $h^{0}_s=k^{0}_s= 0$. By Lemma 3.2, we get immediately that $(U,V) = (0, 0)$.

{\bf Existence} \newline In order to simplify the calculations, we will
always assume that condition $({\bf H2}$-$(ii))$ is satisfied with $\mu =0$. If it is not true, the change
of variables $\tilde{Y}_t = e^{\mu t}Y_t , \tilde{Z}_t = e^{\mu t}Z_t$ reduces to this case. We also split existence into two steps
\newline
{\bf Step 1.} In this part $\xi,\, \sup f^{0}_{t},\,
$ are supposed bounded random variables and $r$ a
positive real such that
\begin{eqnarray*}
e^{(1+\lambda^{2})T}(\|\xi\|_{\infty}+T\|f^{0}
\|_{\infty}) &<& r.
\end{eqnarray*}
Let $\theta_{r}$ be a smooth function such that $0\leq
\theta_{r}\leq 1$ and
\begin{eqnarray*}
\theta_{r}(y)=\left\{ \begin{array}{l}
1\text{ for }|y|\leq r \\
\\
0\text{ for }|y|\geq r+1.
\end{array}
\right.
\end{eqnarray*}
For each $n\in\N^{*}$, we denote $q_{n}(z)=z\frac{n}{|z|\vee n }$ and set
\begin{eqnarray*}
h_{n}(t,y,z)&=&\theta_{r}(y)(f(t,y,q_{n}(z))-f_{t}^{0})\frac{n}{\pi_{r+1}
(t)\vee n}+f_{t}^{0}.
\end{eqnarray*}
Thanks to Pardoux et al. \cite{Pal}, this function still satisfies quadratic condition $({\bf H2}$-$(ii))$ but with a positive constant.
Then data $(\xi,h_{n},g)$ satisfies assumptions of Theorem 4.1. Hence, for each
$n\in\N^*$, backward doubly SDE associated to $(\xi,h_{n},g)$ has a unique solution
$(Y^{n},Z^{n})$ in space $\mathcal{S}^{2}
\times\mathcal{M}^{2}$.

Since
\begin{eqnarray*}
y\ h_{n}(t,y,z)&\leq& |y|\ \|f^{0}\|_{\infty}+\lambda|y|\ |z|,
\end{eqnarray*}
$\xi$  is bounded and $g^{0}_t$ is null, the similar argument used in \cite{BC} (see proposition $2.1$) provide that the process $Y^{n}$ satisfies the inequality $\|Y^{n}\|_{\infty}\leq r$. In addition, from Lemma 3.2,
\begin{eqnarray}
\|Z^{n}\|_{\mathcal{M}^{2}}\leq r'\label{Zborne}
\end{eqnarray}
where $r'$ is another constant. As a byproduct
$(Y^{n},Z^{n})$ is a solution to backward doubly SDE associated to $(\xi,f_{n},g)$ where
\begin{eqnarray*}
f_{n}(t,y,z)&=&(f(t,y,q_{n}(z))-f_{t}^{0})\frac{n}{\pi_{r+1}
(t)\vee n}+f_{t}^{0}
\end{eqnarray*}
which satisfied assumption $({\bf H2}$-$(ii))$ with $\mu=0$.

We now have, for $i\in\N$, setting
$
\begin{array}{l}
\bar{Y}^{n,i}=Y^{n+i}-Y^{n},\ \bar{Z}^{n,i}=Z^{n+i}-Z^{n},
\end{array}
$
applying assumptions $({\bf H2})$ on $f_{n+i}$ and $g$
\begin{eqnarray*}
&&e^{at}|\bar{Y}_{t}^{n,i}|^{2}+(1-\varepsilon-\alpha)\int_{t}^{T}
e^{a s}|\bar{Z}_{s}^{n,i}|^{2}ds \\
&\leq&
2\int_{t}^{T}e^{a s}\langle\bar{Y}_{s}^{n,i},f_{n+i}(s,Y_{s}^{n},Z_{s}^{n})
-f_{n}(s,Y_{s}^{n},Z_{s}^{n})\rangle ds \\
&&+(\frac{1}{\varepsilon}\lambda^{2}+\lambda-a)\int_{t}^{T}e^{a s}|\bar{Y}^{n,i}_{s}|^{2}ds\\
&&+2\int_{t}^{T}e^{a s}\langle\bar{Y}_{s}^{n,i},(g(s,Y^{n+i}_s,Z^{n+i}_s)-g(s,Y^{n}_s,Z^{n}_s))dB_{s}\rangle\\
&&-2\int_{t}^{T}e^{a s}\langle\bar{Y}_{s}^{n,i},\bar{Z}_{s}^{n,i}dW_{s}\rangle,
\end{eqnarray*}
for any $a>0$ and $\varepsilon>0$. Next, choosing $\varepsilon$ small enough such that $\gamma=1-\varepsilon-\alpha>0$ and after $\alpha$ such that
$(\frac{1}{\varepsilon}\lambda^{2}+\lambda-a)\leq 0$, we obtain
\begin{eqnarray*}
&&e^{at}|\bar{Y}_{t}^{n,i}|^{2}+\gamma\int_{t}^{T}
e^{a s}|\bar{Z}_{s}^{n,i}|^{2}ds \\
&\leq&
2\int_{t}^{T}e^{a s}\langle\bar{Y}_{s}^{n,i},f_{n+i}(s,Y_{s}^{n},Z_{s}^{n})
-f_{n}(s,Y_{s}^{n},Z_{s}^{n})\rangle ds \\
&&+(\frac{1}{\varepsilon}\lambda^{2}+\lambda-a)\int_{t}^{T}e^{a s}|\bar{Y}^{n,i}_s|^{2}ds\\
&&+2\int_{t}^{T}e^{a s}\langle\bar{Y}_{s}^{n,i},(g(s,Y^{n+i}_s,Z^{n+i}_s)-g(s,Y^{n}_s,Z^{n}_s))dB_{s}\rangle\\
&&-2\int_{t}^{T}e^{a s}\langle\bar{Y}_{s}^{n,i},\bar{Z}_{s}^{n,i}dW_{s}\rangle.
\end{eqnarray*}
But $\|\bar{Y}^{n,i}\|_{\infty}\leq 2r$ so that
\begin{eqnarray*}
&&e^{at}|\bar{Y}_{t}^{n,i}|^{2}+\gamma\int_{t}^{T}
e^{a s}|\bar{Z}_{s}^{n,i}|^{2}ds \\
&\leq&
4r\int_{t}^{T}e^{a s}|f_{n+i}(s,Y_{s}^{n},Z_{s}^{n})
-f_{n}(s,Y_{s}^{n},Z_{s}^{n})|ds \\
&&+(\frac{1}{\varepsilon}\lambda^{2}+\lambda-a)\int_{t}^{T}e^{a s}|\bar{Y}^{n,i}_s|^{2}ds\\
&&+2\int_{t}^{T}e^{a s}\langle\bar{Y}_{s}^{n,i},(g(s,Y^{n+i}_s,Z^{n+i}_s)-g(s,Y^{n}_s,Z^{n}_s))dB_{s}\rangle\\
&&-2\int_{t}^{T}e^{a s}\langle\bar{Y}_{s}^{n,i},\bar{Z}_{s}^{n,i}dW_{s}\rangle
\end{eqnarray*}
and using successively Gronwall lemma and the BDG inequality, we get, for a constant $C$ depending only on $\lambda,\,\alpha$ and $T$,
\begin{eqnarray*}
&&\E\left[\sup_{0\leq t\leq T}|\bar{Y}_{t}^{n,i}|^{2}+\int_{0}^{T}|\bar{Z}_{s}^{n,i}|^{2}ds\right]\leq
Cr\E\left[\int_{0}^{T}|f_{n+i}(s,Y_{s}^{n},Z_{s}^{n})
-f_{n}(s,Y_{s}^{n},Z_{s}^{n})|ds\right].
\end{eqnarray*}
On the other hand, since $\|Y^{n}\|_{\infty}\leq r$, we get
\begin{eqnarray*}
|f_{n+i}(s,Y_{s}^{n},Z_{s}^{n})-f_{n}(s,Y_{s}^{n},Z_{s}^{n})|
&\leq& 2\lambda|Z^{n}_s|{\bf 1}_{\{|Z^{n}_s|\ >n\}}
+2\lambda|Z^{n}_s|{\bf
1}_{\{\pi_{r+1}(s)>n\}}+2\pi_{r+1}(s){\bf 1}_{\{\pi_{r+1}(s)>n\}}
\end{eqnarray*}
from which we deduce, according assumption $({\bf H3}$-$(iv))$ and
inequality $(\ref{Zborne})$ that $(Y^{n},Z^{n})$ is a cauchy sequence
in the Banach space $\mathcal{S}^{2}\times \mathcal{M}^{2}$. It is easy to pass to the limit in the
approximating equation, yielding a solution to backward doubly SDE $(\ref{a1})$.

{\bf Step 2.}\ We now treat the general
case. For each $n\in \N^{*}$,\ let us define
\begin{eqnarray*}
\xi_{n}= q_{n}(\xi),\,\;\;f_{n}(t,y,z)= f\left(t,y,z\right)-f^{0}_{t}+q_{n}(f^{0}_{t}).
\end{eqnarray*}
For each triplet $(\xi_{n},f_{n},g)$, BSDE (1) has a unique solution $(Y^{n},Z^{n})\in L^{2}$\ thanks to the
first step of this proof, but in fact also in all $L^{p}, p>1$ according to Lemma 3.1. Now from Lemma 3.2 an assumption $({\bf H2})$, for $(i,n)\in \N\times\N^{*}$,
\begin{eqnarray*}
&&\E\left\{\sup_{0\leq t\leq T}|Y_{t}^{n+i}-Y_{t}^{n}|^{p}+
\left(\int_{0}^{T}|Z_{s}^{n+i}-Z_{s}^{n}|^{2}ds\right)^{p/2}\right\}\\
&\leq& C\E\left\{|\xi_{n+i}-\xi_{n}|^{p}+\left(\int_{0}^{T}
|q_{n+i}(f^{0}_{s})-q_{n}(f^{0}_{s})|ds\right)^{p}\right\},
\end{eqnarray*}
where $C$ depends  on $T,\, \alpha$ and $\lambda $.

The right-hand side of the last inequality clearly tends to $0$, as $n\rightarrow\infty$, uniformly in $i$,
so we have again a Cauchy sequence and the limit is a solution to backward doubly SDE $(\ref{a1})$.
\end{proof}

\end{document}